\documentclass[12pt]{amsart}
\usepackage{amsmath,amssymb,amsbsy,amsfonts,amsthm,latexsym,amsopn,amstext,
                                               amsxtra,euscript,,amscd,stmaryrd,mathrsfs,
                                               cite,array,mathtools,color,bm}
                    
\usepackage{url}
\usepackage[colorlinks,linkcolor=blue,anchorcolor=blue,citecolor=blue,backref=page]{hyperref}
\usepackage{color}
\usepackage{mathtools}
\usepackage{todonotes}

\begin{document}

 \renewcommand*{\backref}[1]{}
\renewcommand*{\backrefalt}[4]{%
    \ifcase #1 (Not cited.)%
    \or        (p.\,#2)%
    \else      (pp.\,#2)%
    \fi}

\newtheorem{theorem}{Theorem}
\newtheorem{lemma}[theorem]{Lemma}
\newtheorem{claim}[theorem]{Claim}
\newtheorem{cor}[theorem]{Corollary}
\newtheorem{prop}[theorem]{Proposition}
\newtheorem{definition}{Definition}
\newtheorem{question}[theorem]{Question}
\newtheorem{assump}[theorem]{Assumption}

\numberwithin{equation}{section}
\numberwithin{theorem}{section}
\numberwithin{table}{section}

\def\cA{{\mathcal A}}
\def\cB{{\mathcal B}}
\def\cC{{\mathcal C}}
\def\cD{{\mathcal D}}
\def\cE{{\mathcal E}}
\def\cF{{\mathcal F}}
\def\cG{{\mathcal G}}
\def\cH{{\mathcal H}}
\def\cI{{\mathcal I}}
\def\cJ{{\mathcal J}}
\def\cK{{\mathcal K}}
\def\cL{{\mathcal L}}
\def\cM{{\mathcal M}}
\def\cN{{\mathcal N}}
\def\cO{{\mathcal O}}
\def\cP{{\mathcal P}}
\def\cQ{{\mathcal Q}}
\def\cR{{\mathcal R}}
\def\cS{{\mathcal S}}
\def\cT{{\mathcal T}}
\def\cU{{\mathcal U}}
\def\cV{{\mathcal V}}
\def\cW{{\mathcal W}}
\def\cX{{\mathcal X}}
\def\cY{{\mathcal Y}}
\def\cZ{{\mathcal Z}}

\def\A{{\mathbb A}}
\def\B{{\mathbb B}}
\def\C{{\mathbb C}}
\def\D{{\mathbb D}}
\def\E{{\mathbb E}}
\def\F{{\mathbb F}}
\def\G{{\mathbb G}}
\def\I{{\mathbb I}}
\def\J{{\mathbb J}}
\def\K{{\mathbb K}}
\def\L{{\mathbb L}}
\def\M{{\mathbb M}}
\def\N{{\mathbb N}}
\def\O{{\mathbb O}}
\def\P{{\mathbb P}}
\def\Q{{\mathbb Q}}
\def\R{{\mathbb R}}
\def\S{{\mathbb S}}
\def\T{{\mathbb T}}
\def\U{{\mathbb U}}
\def\V{{\mathbb V}}
\def\W{{\mathbb W}}
\def\X{{\mathbb X}}
\def\Y{{\mathbb Y}}
\def\Z{{\mathbb Z}}

\def\RQ{{\mathsf Q}}

\def\ep{{\mathbf{e}}_p}
\def\em{{\mathbf{e}}_m}
\def\eq{{\mathbf{e}}_q}

\def\scr{\scriptstyle}
\def\\{\cr}
\def\({\left(}
\def\){\right)}
\def\[{\left[}
\def\]{\right]}
\def\<{\langle}
\def\>{\rangle}
\def\fl#1{\left\lfloor#1\right\rfloor}
\def\rf#1{\left\lceil#1\right\rceil}
\def\le{\leqslant}
\def\ge{\geqslant}
\def\eps{\varepsilon}
\def\mand{\qquad\mbox{and}\qquad}

\def\sssum{\mathop{\sum\ \sum\ \sum}}
\def\ssum{\mathop{\sum\, \sum}}
\def\ssumw{\mathop{\sum\qquad \sum}}

\def\vec#1{\mathbf{#1}}
\def\inv#1{\overline{#1}}
\def\num#1{\mathrm{num}(#1)}
\def\dist{\mathrm{dist}}

\def\fA{{\mathfrak A}}
\def\fB{{\mathfrak B}}
\def\fC{{\mathfrak C}}
\def\fU{{\mathfrak U}}
\def\fV{{\mathfrak V}}

\newcommand{\bflambda}{{\boldsymbol{\lambda}}}
\newcommand{\bfxi}{{\boldsymbol{\xi}}}
\newcommand{\bfrho}{{\boldsymbol{\rho}}}
\newcommand{\bfnu}{{\boldsymbol{\nu}}}

\def\GL{\mathrm{GL}}
\def\SL{\mathrm{SL}}

\def\Hba{\overline{\cH}_{a,m}}
\def\Hta{\widetilde{\cH}_{a,m}}
\def\Hb1{\overline{\cH}_{m}}
\def\Ht1{\widetilde{\cH}_{m}}

\def\bfell{\boldsymbol{\ell}}
\def\vx{\vec{x}}
\def\vy{\vec{y}}

\def\flp#1{{\left\langle#1\right\rangle}_p}
\def\flm#1{{\left\langle#1\right\rangle}_m}
\def\dmod#1#2{\left\|#1\right\|_{#2}}
\def\dmodq#1{\left\|#1\right\|_q}

\def\Zm{\Z/m\Z}

\def\Err{{\mathbf{E}}}

\newcommand{\commB}[2][]{\todo[#1,color=green!60]{B: #2}}
\newcommand{\commI}[2][]{\todo[#1,color=yellow]{I: #2}}
\newcommand{\commR}[2][]{\todo[#1,color=red]{R: #2}}

\def\ccr#1{\textcolor{red}{#1}}
\def\cco#1{\textcolor{orange}{#1}}
\def\ccc#1{\textcolor{cyan}{#1}}

\def\xxx{\vskip5pt\hrule\vskip5pt}


\title[Bohr Sets Generated by Polynomials]{\bf Bohr Sets Generated by Polynomials and Coppersmith's method in many variables}

 \author[R. Baird] {Riley Baird}
\address{School of Science, University of New South Wales,
Canberra, ACT 2106, Australia}
\email{riley@mailo.com}

 \author[B. Kerr] {Bryce Kerr}

\address{School of Science, University of New South Wales,
Canberra, ACT 2106, Australia}
\email{bryce.kerr@unsw.edu.au}

 \author[I. E. Shparlinski] {Igor E. Shparlinski}

\address{Department of Pure Mathematics, University of New South Wales,
Sydney, NSW 2052, Australia}
\email{igor.shparlinski@unsw.edu.au}

\date{\today}
\pagenumbering{arabic}

\begin{abstract} We obtain bounds on the average size of Bohr sets with
coefficients parametrised by polynomials over finite fields and obtain a series of 
general results and also some sharper results for specific sets which are 
important for applications to computer science. In particular, we use our estimates to show that a heuristic assumption used in the many variable version of Coppersmith's method holds with high probability. We demonstrate the use of our results on the approximate greatest common divisor
problem and obtain a fully rigorous version of the heuristic algorithm of H.~Cohn and N.~Heninger (2013).
\end{abstract}

\maketitle

\tableofcontents

\section{Introductoion}

\subsection{Bohr sets}

Let $\cG$ be a	commutative	group.	Given $n$ characters $\chi_1, \ldots, \chi_n$
of $\cG$ and $n$ real numbers $\xi_1, \ldots, \xi_n \in  (0, 1/2]$, we define the 
{\it Bohr set\/}
\begin{equation*}
\begin{split}
\fB(\chi_1, & \ldots, \chi_n;\xi_1, \ldots, \xi_n) \\
&= 
\{g \in  \cG:~|\arg \chi_j(g)| \le 2\pi \xi_j, \ j = 1,  \ldots,  n\},
\end{split}
\end{equation*}
where we take the branch of $\arg$ that lies in $[-\pi, \pi)$. 

For $q$ prime let $\F_q$ denote the
finite field of $q$ elements which we assume to be represented 
by the set of integers $\{0, 1, \ldots, q-1\}$. For $\cG = \F_q$ one usually uses the following equivalent 
but more convenient definition of Bohr sets. 

Given vectors 
$$
\vec{a}=(a_1,  \ldots,  a_n) \in \Z^n \mand \vec{h} = (h_1,  \ldots,  h_n) \in \N^n
$$ 
we define the  {\it Bohr set modulo $q$\/} as
$$
\fB_q(\vec{a}, \vec{h}) 
= 
\{s \in  \F_q:~  \|a_js/q\| \le h_j/q, \ j = 1,  \ldots, n\},
$$
where  
$$
\|\zeta \| = \min\{\{\zeta\}, 1-\{\zeta\}\}
$$ 
is the distance between a real number $\zeta$ and its closest integer.

Clearly any Bohr set $\fB_q(\vec{a}, \vec{h})$ contains $u=0$, and 
if $\fB_q(\vec{a}; \vec{h}) = \{0\}$ we say that it is {\it trivial\/}. 
A simple counting argument shows that if $\vec{h}$ is fixed and 
\begin{equation}
 \label{eq:Prod h}
h_1\ldots h_n = o(q^{n-1}),  
\end{equation}
then for  all but at most $o(q^{n})$  vectors 
$\vec{a}  \in \F_q^n$ the Bohr set $\fB_q(\vec{a}; \vec{h}) $ 
is trivial. 

Here we investigate the question of triviality for the parametric family
of Bohr sets  
$$\fB_q(\vec{f};\vec{h};u)=  \fB_q((f_1(u), \ldots, f_n(u)); \vec{h}),
$$
where $f_1 , \ldots, f_n \in \F_q[X]$ are   $n$ linearly independent polynomials.

Our purpose is to investigate various conditions on 
$\vec{h}$ which imply that $\fB_q(\vec{f}; \vec{h};u)$ is
trivial for almost all $u \in \F_q$. Our motivation for this problem comes from the multivariable Coppersmith method and its various applications to encryption and coding theory which we discuss in Section~\ref{sec:appl}.

\subsection{Notation}

We define $\fU_q(\vec{f}; \vec{h})$ as the set 
of $u \in \F_q$ for which $\fB_q(\vec{f}; \vec{h};u)$ is trivial.
We are interested in showing that  
$\# \fU_q(\vec{f}; \vec{h}) = o(q)$ provided that 
essentially the natural condition~\eqref{eq:Prod h} holds
and maybe also some other conditions, which we try to
keep as generous as possible.  
Here we concentrate on the special case of  monomials 
\begin{equation}
\label{eq:monom}
\vec{f}_{\vec{a}, \vec{k}} = \( a_1X^{k_1},  \ldots,  a_nX^{k_n}\),
\end{equation}
with $\vec{a}, \vec{k} \in \Z^n$ such that 
\begin{equation}
\label{eq:vec a k}
a_i \not \equiv 0 \mod  q,\quad  k_i > 0, \quad  k_i \ne k_j, \qquad 
1 \le i, j \le n, \ i \ne j.
\end{equation}
Permuting, if necessary, the order of the monomials in 
$\vec{f}_{\vec{a}, \vec{k}}$, we always assume that 
\begin{equation}
\label{eq:order h}
1\le h_1 \le \ldots \le h_n. 
\end{equation}
We note that the assumptions~\eqref{eq:vec a k} are not necessary for our techniques but the 
final bounds are stronger under these assumptions. 
In this case we use 
$\cU_{n,q}(\vec{a}, \vec{k}, \vec{h})$ 
to denote the set $\fU_q(\vec{f}_{\vec{a}, \vec{k}}; \vec{h})$. Furthermore when 
 $k_j = j$,  we simply write 
 $\cU_{n,q}(\vec{a},\vec{h})$.

Throughout, any implied constants in symbols $O$, $\ll$
and $\gg$ may  depend  on $n$ and  the degrees of 
the polynomials $f_1, \ldots, f_n$ (and occasionally on 
other explicitly specified parameters). 
 We recall
that the notations $A = O(B)$,  $A \ll B$ are equivalent to the statement that $|A| \le c B$ holds 
with some absolute constant $c> 0$.

\subsection{Reformulation and simple estimates}
Here we present some general observations which apply to 
arbitrary polynomials. 

Clearly  $\fU_q(\vec{f}; \vec{h})$ is the set of $u\in \F_q$ 
for which the following system of congruences
\begin{equation}
\label{eq:congrf}
sf_j(u) \equiv x_j \mod  q, \qquad   s \in \F_q^*, \ |x_j|\le h_j, 
\quad j = 1, \ldots, n.
\end{equation}
has a solution.
 
Denoting 
$$
d = \max\{\deg f_1, \ldots, \deg f_n\}, 
$$
we see that~\eqref{eq:congrf} may have $x_1 \ldots x_n = 0$ for at most  
$dn$ values of $u$. For the remaining $u$, we use 
the first congruence in~\eqref{eq:congrf} to eliminate 
$s$ from the system~\eqref{eq:congrf}. This leads us to the inequality
\begin{equation}
\label{eq:U V}
\# \fU_q(\vec{f}; \vec{h}) \le V_p(\vec{f}; \vec{h}) +O(1),
\end{equation}
where $V_q(\vec{f}; \vec{h})$ is the number of solutions of
the following system of congruences
$$
x_n R_j(u) \equiv x_j \mod  q, \qquad  u \in \F_q^*, \ 1 \le |x_j|\le h_j, 
\quad j = 1, \ldots, n,
$$
with rational functions 
$$
R_j(X) =\frac{f_j(X)}{f_1(X)},\qquad j = 1, \ldots, n. 
$$
Certainly the last congruence (with $j=1$)
gives us no useful information. Discarding it, we
see that for each fixed $x_1$ the resulting system of congruences
counts the number of times several rational functions 
fall simultaneously in prescribed intervals modulo $q$.

There is extensive literature which studies questions of this kind 
for  polynomials and rational functions in one or several variables, see~\cite[Theorem~21.4]{IwKow} for a typical result of this type.

Clearly,  since the polynomials $f_1, \ldots, f_n$ are 
linearly independent, any nontrivial linear combination of the functions 
$R_2, \ldots, R_{n}$ is never constant.
Hence, a standard 
application of the {\it Weil bound} (see, for example,~\cite{CoPi,MorMor})
for exponential sums with rational functions implies that 
$$
\sum_{0 <|x| \le h} \sum_{u \in \F_q} \hskip-18pt \phantom{\sum}^{*}\,
\eq\(x\sum_{j=2}^n \lambda_j R_j(u)\)
= O(h q^{1/2}),
$$
for any nonzero vector $(\lambda_2, \ldots, \lambda_n) \in \F_q^{n-1}$,
where $\Sigma^*$ denotes the poles of the functions $R_2,\dots,R_n$ are excluded from summation and we let 
$$
\eq(z) = \exp(2\pi i z/q).
$$
Now,  the {\it Koksma--Sz\"usz inequality\/} (see~\cite[Theorem~1.21]{DrTi}) 
 immediately implies that 
$$
V_q(\vec{f}; \vec{h}) = \frac{2^n} {q^{n-2}} \prod_{j=1}^n h_j  + O(h_1 q^{1/2}(\log q)^{n-1}).
$$
So, if $h_1 \le q^{1/2-\varepsilon}$ for some fixed $\varepsilon >0$
and if~\eqref{eq:Prod h} holds, then recalling~\eqref{eq:U V},
we obtain the desired bound $\# \fU_q(\vec{f}; \vec{h})  = o(q)$.
The condition on $h_1$ can be slightly relaxed, but generally the 
above result seems to be the limit of this approach. 

For specific families of polynomials which arise naturally from problems in computer science (see Section~\ref{sec:appl} for details) one can 
obtain stronger results by analysing the distribution of points on modular hyperbolas.

\section{Main Results}
\label{sec:main}
\subsection{Results for monomials} 

First we obtain a result that holds for every prime $q$. 

\begin{theorem}
\label{thm:main1}
Assume that  $n \ge 3$. Let $\vec{a}, \vec{k}$ satisfy~\eqref{eq:vec a k}
and $\vec{h}$ satisfy~\eqref{eq:order h}. Then
$$
\# \cU_{n,q}(\vec{a}, \vec{k}, \vec{h})\ll (h_1h_2h_3)^{1/2}\log q.
$$
\end{theorem}

Note that 
$$
h_1h_2h_3 \le (h_1\ldots h_n)^{3/n}.
$$
Hence  Theorem~\ref{thm:main1} implies 
that $\# \cU_{n,q}(\vec{a}, \vec{k}, \vec{h}) = o(q)$ provided 
that 
$$
h_1\ldots h_n  \le q^{2n/3 - \varepsilon},
$$
for some fixed $\varepsilon > 0$. However, 
in our applications of Theorem~\ref{thm:main1} 
the  sizes of $h_1, \ldots, h_n$ are very different
and thus we have  $\# \cU_{n,q}(\vec{a}, \vec{k}, \vec{h}) = o(q)$, 
despite the product $h_1\ldots h_n$ being
close to the threshold $q^{n-1}$. 

We now consider the special case  $f_j(X) = X^{j}$, $j=1, \ldots, n$.
\begin{theorem}
\label{thm:main12}
Let $n \ge 3$. Then   
we have 
$$
\# \cU_{n,q}(\vec{a},\vec{h})\le \(h_1h_2h_3/q + h_2^3/q + h_2\) \exp{\left(O\left(\frac{\log{q}}{\log\log{q}}\right)\right)}.
$$
\end{theorem}

Our second result shows that on average over primes we can improve on 
Theorem~\ref{thm:main1}.

\begin{theorem}
\label{thm:main2}
Assume that  $n \ge 3$. Let $\vec{k}$ and $\vec{h}$ satisfy~\eqref{eq:vec a k}
and ~\eqref{eq:order h}, respectively. Suppose that $\vec{a} \in \Z^n$ is
such that 
$$
0 < |a_i| \le A, \qquad i=1, \ldots, n,
$$
for some real $A$. 
Then for any set of primes $\cQ$, 
we have
\begin{align*}
&\frac{1}{\#\cQ}\sum_{\substack{q\in \cQ}}\# \cU_{n,q}(\vec{a}, \vec{k}, \vec{h}) \\ &\qquad \quad \ll
h_3\exp\left({O\left(\frac{\log{A}}{\log\log{A}}+\log\log{h_3}\right)}\right) +\frac{h_1h_2h_3 \log h_3}{\#\cQ\log\log h_3}.
\end{align*}
\end{theorem}

When $k_i = i$, $i = 1,\ldots,n$, we are able to get a 
slightly stronger version of Theorem~\ref{thm:main2} (with $h_2$ 
in the first term instead of $h_3$). 

\begin{theorem}
\label{thm:main3}
Assume that  $n \ge 3$ and
let $\vec{h}$  satisfy~\eqref{eq:order h}. 
 Then for any set of primes $\cQ$, 
we have
$$
\frac{1}{\#\cQ}\sum_{\substack{q\in \cQ}}\# \cU_{n,q}(\vec{h}) \ll
h_2\log^2{h_2} 
+\frac{h_1h_2h_3 \log h_3}{\#\cQ\log\log h_3}.
$$
\end{theorem}

\subsection{Results for special polynomials} 
We next consider a special family of polynomials which are important for applications discussed in Section~\ref{sec:appl}. Let  $\ell$ and $m$ be integers and let $r=(r_1,\ldots,r_m)$ be a $k$-tuple of integers with each $1\le r_i<\ell$. 
For an $m$-tuple of 
nonnegative
integers $\vec{e}=\(e_1,\ldots,e_m\)$ we write
$$|\vec{e}|=\sum_{i=1}^{m}e_i,$$
 and for each $\vec{e}=\{e_1,\ldots,e_m\}$ with   $1\le |\vec{e}| \le t$ we define the polynomials
\begin{equation}
\label{eq:fdef}
f_{\vec{e}}(y_1,\ldots,y_m)=\frac{1}{\ell} \(\prod_{i=1}^{m}\left(\ell y_i+r_i \right)^{e_i}-\prod_{i=1}^{m}r_i^{e_i}\).
\end{equation}

For $q$ prime and a tuple of integers  $\vec{X}=\{X_{\vec{e}}\}_{1\le |\vec{e}|\le t}$,  let $\cU^{(m)}_{q}(\vec{X})$ denote the set of solutions to the system of equations
$$sf_{\vec{e}}(y_1,\ldots,y_m)\equiv x_{\vec{e}} \mod  q, \quad 1\le |\vec{e}| \le t,$$
in variables $s,y_i,x_{\vec{e}}$ satisfying
$$1\le s \le q-1, \quad 1\le y_i \le q-1, \quad |x_{\vec{e}}|\le X_{\vec{e}}, \quad 1\le |\vec{e}|\le t.$$
We first bound  $\#\cU^{(m)}_q(\vec{X})$ for a single prime $q$.
\begin{theorem}
\label{thm:approxgcd}
With notation as above, let $H,\ell >0$ be integers and suppose each $X_{\vec{e}}$ is given by
$$X_{\vec{e}}=\frac{qH^{|\vec{e}|}}{\ell}.$$
Then if each $r_i\ll H$ we have 
$$
\#\cU_q^{(m)}(\vec{X})\le \left(\frac{q^2H^6}{\ell^3}+\frac{qH^2}{\ell}\right)\left(\frac{qH}{\ell}\right)^{m-1}\exp{\left(O\left(\frac{\log{q}}{\log\log{q}}\right)\right)}.
$$ 
\end{theorem} 

Our next result shows we can improve on Theorem~\ref{thm:approxgcd} on average.

\begin{theorem}
\label{thm:approxgcdav}
With notation as above, let $H,Q,\ell >0$ be  integers and suppose each $X_{\vec{e}}$ is given by
$$X_{\vec{e}}=\frac{QH^{|\vec{e}|}}{\ell}.$$
Let $\cQ\subseteq [Q,2Q]$ be a  set of primes. 
Then if each $|r_i|\ll H$ we have
\begin{align*}
\frac{1}{\#\cQ}\sum_{q\in \cQ}& \#\cU_q^{(m)}(\vec{X}) \\
&\ll \left(\frac{QH}{\ell}\right)^{m-1}\left( 
\frac{QH^2}{\ell}\log^2{Q} 
+\frac{1}{\#\cQ}\frac{Q^3H^6}{\ell^3}\frac{\log Q}{\log\log Q}\right).
\end{align*}
\end{theorem}

\section{Applications}
\subsection{The approximate greatest  common divisor
problem}
\label{sec:appl}
In this section we motivate the results obtained in Section~\ref{sec:main} by giving an application to the approximate greatest common divisor problem.  Given integers $N,X_1,\ldots,X_m,a_1,\ldots,a_m$ and a real number $\beta$, the partial approximate common divisor problem is to determine an algorithm which runs in polynomial time with respect to the lengths of inputs and determines all integers $r_1,\ldots,r_m$ such that 
$$\gcd(N,a_1-r_1,\ldots,a_m-r_m)\ge N^{\beta}, \quad |r_i|\le X_i, \quad i=1,\ldots,m.$$
The general approximate common divisor problem has a similar setup as above but seeks to determine all integers $r_1,\ldots,r_m$ such that 
$$\gcd(a_1-r_1,\ldots,a_m-r_m)\ge N^{\beta}, \quad |r_i|\le X_i, \quad i=1,\ldots,m.$$
These problems have origins in Coppersmith's method~\cite{Cop} and the case  $m=2$ was first considered by Howgrave-Graham~\cite{HG} with various applications to cryptography.  The general case of $m\ge 2$ has been considered by Cohn and Heninger~\cite{CH} with further applications to cryptography and coding theory.  The approach of~\cite{CH} is subject to a heuristic assumption which is observed to hold in practice (see~\cite[Section~2]{CH}) 
but so far has lacked any theoretical explanation. Our main application of the results from Section~\ref{sec:main} is to show that the heuristic assumption used in~\cite{CH} holds with a high probability.

Namely, we have the following result which shows that for an overwhelming majority of inputs there is 
a fully rigorous, deterministic polynomial time algorithm to solve the approximate greatest common divisor
problem.

\begin{theorem}
\label{thm:ApproxGCD-Rigor}
Let $m$ be an integer, $\varepsilon>0$ be small and $a_0=pq$ with $p,q$ prime. Suppose $H$ satisfies  
$$
H=O\left(p^{1-1/(m+1)-\varepsilon}\right).
$$
 If the tuple $(a_1,\ldots,a_m)$ is chosen uniformly at random from the set 
$$\{(pb_1+r_1,\ldots,pb_m+r_m) :~ 1\le b_j\le q, \ r_j\le H\},$$ 
then with probability $1+o(1)$ as $p,q,H \to \infty$,
 there exists a deterministic polynomial time algorithm to calculate all 
integers $r_1,\ldots,r_m$ satisfying 
$$\gcd(a_0,a_1-r_1,\ldots,a_m-r_m)\ge p.$$
\end{theorem}

 To prove  Theorem~\ref{thm:ApproxGCD-Rigor},  we first present the relevant background from~\cite{CH}. 
 In particular, see  Lemma~\ref{lem:approxgcd-alg} which describes Cohn and Heninger's~\cite{CH} conditional algorithm. 
We then establish our main tool,  Lemma~\ref{lem:dual}, which allows us to get an unconditional algorithm 
and hence derive Theorem~\ref{thm:ApproxGCD-Rigor} in Section~\ref{sec: ApproxGCD-Rigor}.

\subsection{Outline of the Cohn and Heninger method}
\label{sec:applications}
We concentrate on the case of the partial approximate common divisor problem. Given integers $a_0,a_1,\ldots,a_m$ we seek to calculate all $p,r_1,\ldots,r_m$ such that 
\begin{equation}
\label{eq:rpcond}
 a_0=pq, a_1=pq_1+r_1,\ldots,a_m=pq_m+r_m, \quad r_i\le X_i, \quad p\ge a_0^{\beta}.
\end{equation}
Assume $p,q$ are both prime. We take some parameters $t$ and $k$ to be determined later and for each $m$-tuple of integers $e=(e_1,\ldots,e_m)$ satisfying 
\begin{equation}
\label{eq:econd}
1\le e_1+\ldots+e_m\le t
\end{equation}
 we define the polynomial
\begin{equation}
\label{eq:fedef}
f_\vec{e}(\vec{x})=\prod_{i=1}^{m}(X_ix_i+a_i)^{e_i}a_0^{\ell},
\end{equation}
where $\vec{e}=(e_1,\ldots,e_m)$, $\vec{x}=(x_1,\ldots,x_m)$ and 
\begin{equation}
\label{eq:elldef}
\ell=\max{\{k-(e_1+\ldots+e_m),0\}}.
\end{equation}
Let $\cL$ denote the lattice formed by coefficients of the polynomials $f_\vec{e}$ with $e$ satisfying~\eqref{eq:econd}. In particular, we may consider $\cL$ as a subset of Euclidian space via a lexicographic ordering of coordinates. For an ordered tuple of real numbers 
$$\vec{y}=\{y_{i_1,\dots,i_m}\}_{i_1+\dots+i_m\le t},$$
we associate the polynomial
\begin{align}
\label{eq:Q123}
Q_{\vec{y}}(x_1,\dots,x_m)=\sum_{i_1+\dots+i_m\le t}y_{i_1,\dots,i_m}x_1^{i_1}\dots x_m^{i_m}.
\end{align}

If 
$$r_i\equiv a_i \mod{p}, \qquad 1\le i \le m,$$
then for each $e_1,\ldots,e_m$ satisfying~\eqref{eq:econd} we have 
$$f_\vec{e}\left(\frac{r_1}{X_1},\ldots,\frac{r_m}{X_m}\right)\equiv 0 \mod{p^k},$$
and hence for each point $\vec{y}\in \cL$ there exists a polynomial $Q_\vec{y}$ as above, such that 
$$Q_{\vec{y}}\left(\frac{r_1}{X_1},\ldots,\frac{r_m}{X_m}\right)\equiv 0 \mod{p^k},$$
whenever $r_1,\ldots,r_m$ satisfy~\eqref{eq:rpcond}. Note if $\vec{y}\in \cL$ then the polynomials  
$$Q_{\vec{y}}\left(\frac{x_1}{X_1},\ldots,\frac{x_m}{X_m}\right),$$
have integral coefficients and a straightforward calculation shows that
\begin{equation}
\label{eq:dimdetL}
\begin{split}
& \dim{\cL}=\binom{t+m}{m}, \\
&  \det{\cL}=(X_1\ldots X_m)^{\binom{t+m}{m}t/(m+1)}a_0^{\binom{k+m}{m}k/(m+1)}.
\end{split} 
\end{equation}

The next step in~\cite{CH} is to apply the LLL-algorithm of 
Lenstra, Lenstra
and  Lov{\'a}sz~\cite{LLL} (see also~\cite{Ngu,NgSt1,NgSt2,Reg}) 
 to $\cL$ which finds $m$ short, linearly independent lattice points. 
 Let $\RQ_1,\ldots,\RQ_m$ denote the polynomials corresponding (as in~\eqref{eq:Q123}) to these $m$ short linearly independent lattice points of $\cL$, so that
$$
\RQ_i\left(\frac{r_1}{X_1},\ldots,\frac{r_m}{X_m}\right)\equiv 0 \mod{p^k}, \quad 1\le i \le m,
$$
and for   $1\le i \le m$ we have 
$$
\left|\RQ_i\left(\frac{r_1}{X_1},\ldots,\frac{r_m}{X_m}\right)\right|\le (\dim{\cL})^{1/2}2^{\dim{\cL}/4}(\det{\cL})^{1/(\dim{\cL}+1-m)}.
$$
Using~\eqref{eq:dimdetL}, and assuming our parameters satisfy
$$
\(\dim{\cL}\)^{1/2}2^{\dim{\cL}/4}(\det{\cL})^{1/(\dim{\cL}+1-m)}<a_0^{k\beta},
$$
we see that if $a_0,r_1,\ldots,r_m$ satisfy~\eqref{eq:rpcond} for some $p\ge a_0^\beta$ then
\begin{equation}
\label{eq:Qieqn}
\RQ_i\left(\frac{r_1}{X_1},\ldots,\frac{r_m}{X_m}\right)=0, \quad 1\le i \le m.
\end{equation}
One then solves the system of polynomial equations~\eqref{eq:Qieqn} to obtain possible candidates for $r_1,\ldots,r_m$ from which solutions to the equations~\eqref{eq:rpcond} can be tested via the Euclidian algorithm. In order for this last stage to be computationally feasable, a heuristic assumption that the polynomials $\RQ_1,\ldots,\RQ_m$ are algebraically independent is used, since then Bezout's theorem implies the system~\eqref{eq:Qieqn} has at most $O(1)$ solutions. 

A careful analysis of the above argument results in the following conditional result which is essentially~\cite[Theorem~1]{CH} (and we present here in a self-contained form, suitable for 
our applications). 

\begin{assump}
\label{ass:heuristic}
The $m$ shortest points of the lattice $\cL$ correspond 
to algebraically independent polynomials. 
\end{assump}

Then by~\cite[Theorem~1]{CH} we have.

\begin{lemma}
\label{lem:approxgcd-alg}
Given positive integers $a_0,a_1,\ldots,a_m$ and real numbers $\beta,X_1,\ldots,X_m$ satisfying
$$
\beta \gg (\log{N})^{-1/2} \mand (X_1\ldots X_m)^{1/m}<a_0^{(1+o(1))\beta^{(m+1)/m}},
$$
there exists a deterministic algorithm which, provided that Assumption~\ref{ass:heuristic} holds, runs in polynomial time and finds all integers $r_1,\ldots,r_m$ satisfying
$$\gcd(a_0,a_1-r_1,\ldots,a_m-r_m)\ge a_0^{\beta} \mand |r_i| \le X_i, \ i =1, \ldots, m$$
\end{lemma}

\subsection{Preliminary discussion}
In what follows, we show that Theorem~\ref{thm:approxgcdav} implies that heuristic 
Assumption~\ref{ass:heuristic} used in Lemma~\ref{lem:approxgcd-alg} holds with a high probability if $a_1,\ldots,a_m$ are selected uniformly at random from integers of bounded height.

Consider the lattice $\cL$ described in Section~\ref{sec:applications}. We   show that in case $\ell=1$, an LLL-reduced basis of the lattice $\cL$ described at the beginning of Section~\ref{sec:applications} behaves like a random lattice with high probability. An equivalent way to state this is that the first successive minima of $\cL$ is very small and all the remaining successive minima have about the same size.  Note that this provides a theoretical explanation for the phenomena observed in~\cite[page~9]{CH} that even when the polynomials obtained from the $m$ smallest lattice points of $\cL$ are algebraically dependent, it is still possible to choose $m$ `short enough' vectors which correspond to algebraically independent polynomials.

Our approach to estimate the successive minima of the lattices $\cL$ on average is to show that points in the dual lattice $\cL^{*}$ correspond to Bohr sets generated by polynomials. This will imply that on average the first successive minima of $\cL^{*}$ is large which combined with transference theorems (see~\cite{Ba}) implies that $\cL$ has many small linearly independent lattice points.

Recall that given a lattice $\cL\subseteq \R^{d}$, the dual lattice $\cL^*$  is defined by 
\begin{equation}
\label{eq:dualdef}
\cL^*=\{ x\in \R^{d} :~ \langle x,y\rangle\in \Z, \quad \text{for all $y\in \cL$} \}.
\end{equation}
Given a convex body $B$, the dual body $B^{*}$ is defined by 
\begin{equation}
\label{eq:dualBdef}
B^*=\{ x\in \R^{d} :~ |\langle x,y\rangle|\le 1 \  \text{for all $y\in B$} \}.
\end{equation}

For a proof of the following, see~\cite{Ba}.
\begin{lemma}
\label{lem:ba}
Let $\cL\subseteq \R^{d}$ be a lattice and $B$ a convex body. Let $\lambda_d$ denote the $d$-th successive minima of $\cL$ with respect to $B$ and $\lambda_1^{*}$ denote the first successive minima of $\cL^{*}$ with respect to $B^{*}$.  We have 
$$\lambda_d \lambda^{*}_1 \ll 1.$$
\end{lemma}

\subsection{Connection to Bohr sets generated by polynomials}
\label{sec:bohrappl}
Returning to the lattices described in Section~\ref{sec:applications}, we first perform a linear change of variables and let $\cL_0$ denote the lattice generated by coefficients of the polynomials
\begin{equation}
\begin{split}
\label{eq:f-expanded}
f_\vec{e}&\left(\frac{x_1}{X_1},\ldots,\frac{x_m}{X_m}\right)\\
& \qquad =a_0^{\ell}\sum_{\substack{0\le j_i\le e_i \\ 1\le j \le m}}\binom{e_1,\ldots,e_m}{j_1,\ldots,j_m}a_1^{e_1-j_1}\ldots a_m^{e_m-j_m}x_1^{j_1}\ldots x_m^{j_m},
\end{split} 
\end{equation}
with $e_1,\dots,e_m$ satisfying 
$$1\le e_1+\ldots+e_m\le t,$$
and 
$$\binom{e_1,\ldots,e_m}{j_1,\ldots,j_m}=\binom{e_1}{j_1}\ldots \binom{e_m}{j_m}.$$

We use a natural correspondence between coefficients of $f_\vec{e}(\vec{x})$ and coordinates of 
points in $\R^{\binom{t+m}{m}}.$
In particular, we order the monomials occuring in~\eqref{eq:f-expanded} lexicographically. Note that $\dim{\cL_0}=\binom{t+m}{m}$ and with the above convention, each point $ \vec{x}\in \R^{\binom{t+m}{m}}$ has a representation 
$$ \vec{x}=\{x_{j_1,\ldots,j_m}\}_{j_1+\ldots+j_m\le t}.$$

Let $B$ denote the box 
$$
B=\left\{ \vec{x}=\{x_{j_1,\ldots,j_m}\}_{j_1+\ldots+j_m\le t} :~ |x_{j_1,\ldots,j_m}|\le \frac{1}{X_1^{j_1}\ldots X_m^{j_m}} \right\}.
$$
Note that the successive minima of $\cL_0$ with respect to $B$ equal the successive minima of $\cL$ with respect to the unit cube $[-1,1]^{\binom{t+m}{m}}$.

Our next result gives a correspondence between points of the dual lattice $\cL^{*}_0$ and Bohr sets generated by polynomials.  We establish a description of the dual $\cL_0^{*}$ in greater generality than needed for our purpose since we expect further applications of the ideas discussed below.

\begin{lemma}
\label{lem:dual}
With notation as above,  we have
$$
\vy=\{y_{j_1,\ldots,j_m}\}_{j_1+\ldots+j_m\le t}\in \cL_0^{*}
$$ 
if and only if there exists integers $k_{\gamma_1,\ldots,\gamma_m}$ such that 
\begin{align*} 
y_{j_1,\ldots,j_m}=\frac{1}{a_0^u} \sum_{\substack{\gamma_{\ell}\le j_{\ell}\\ 1 \le \ell \le m}} (-1)^{\sum_{i=1}^m (j_i - \gamma_i)} & a_0^{\min\left\{ u, \sum_{i=1}^m \gamma_i \right\}}\\
& k_{\gamma_1, \cdots, \gamma_m} \prod_{i=1}^{n} \binom{j_i}{\gamma_i} a_i^{(j_i - \gamma_i)}.
\end{align*}
\end{lemma}

\begin{proof}
We   proceed by induction on $j_1+\ldots+j_m$ and formulate our induction hypothesis as follows: Let $\vy=\{y_{j_1,\ldots,j_m}\}_{j_1+\ldots+j_m\le t}\in \cL_0^{*}$. Suppose $t\ge 0$. There exists integers $k_{j_1,\ldots,j_m}$ such that for all $\beta_1,\ldots,\beta_m$ satisfying 
$$
\sum_{i=1}^{m}\beta_i \le t,
$$
we have 
\begin{equation}
\label{eq:yyy1}
\begin{split} 
y_{\beta_1,\ldots,\beta_m}=\frac{1}{a_0^u} \sum_{\substack{\gamma_{\ell}\le \beta_{\ell} \\ 1 \le \ell \le m}}  (-1)^{\sum_{i=1}^n (j_i - \beta_i)} & a_0^{\min\left\{ u, \sum_{i=1}^m \gamma_i \right\}}\\
& k_{\gamma_1, \cdots, \gamma_m} \prod_{i=1}^{m} \binom{\beta_i}{\gamma_i} a_i^{(\beta_i - \gamma_i)}.
\end{split}
\end{equation}
 We  consider the following two cases separately 
\begin{equation}
\label{eq:dualcase1}
u \ge \beta
\end{equation}
and
\begin{equation}
\label{eq:dualcase2}
u < \beta.
\end{equation}
Note that the base casis $t=0$ is a direct consequence of the definition of lattice dual. In particular, since the point 
$\widetilde{\vec{x}}=\{ x_{j_1,\ldots,j_m}\}_{j_1+\ldots+j_m\le t}$ with coordiates 
satisfying 
$$
x_{j_1,\ldots,j_m}=\begin{cases}
a_0^{u} \quad \text{if} \quad  j_1+\ldots+j_m=0, \\ 0 \quad \text{otherwise,}
\end{cases}
$$
belongs to $\cL_0$, we see that there exists $k_{0,\ldots,0}\in \Z$ such that 
$$y_{0,\ldots,0}=\frac{k_{0,\ldots,0}}{a_0^{u}}.$$

We next establish some notation which  is used throughout the inductive step. Let
\begin{equation}
\label{eq:Idef}
\begin{split} 
& \cI_{\alpha_1, \cdots, \alpha_m} = \left\{ (\gamma_1, \cdots, \gamma_m):~ 0 \le 
\gamma_i \le \alpha_i, \ 1 \le i \le m\right\},\\
& \cI^*_{\alpha_1, \cdots, \alpha_m} = \cI_{\alpha_1, \cdots, \alpha_m} \backslash \{(\alpha_1, \cdots, \alpha_m)\}
\end{split} 
\end{equation}
and let $\bfell_{\alpha_1,\ldots,\alpha_m}$ denote the vector corresponding to the coefficients of the polynomials
\begin{equation}
\begin{split} 
\label{eq:lcoefficients}
& a_0^{\max\left\{0, u-\sum_{i=0}^m \alpha_i\right\}} (a_1 + x_1)^{\alpha_1} \cdots (a_n + x_m)^{\alpha_m} \\ 
&\qquad \quad =a_0^{u - \sum_{i=0}^m \alpha_i} \sum_{(\gamma_1, \cdots, \gamma_m) \in \cI_{\alpha_1, \cdots, \alpha_m}} \prod_{i=1}^{m} \binom{\alpha_i}{\gamma_i} a_i^{\alpha_i-\gamma_i} x_i^{\gamma_i}. 
\end{split}
\end{equation}
It is clear that $\mathcal{L}_0$ is equivalent to the lattice generated by vectors $\bfell_{\alpha_1, \cdots, \alpha_m}$ with $\alpha_1,\dots,\alpha_m$ satisfying $\alpha_1+\ldots+\alpha_m\le t$ and hence 
$$\vy\in \cL_0^{*} \iff \langle \vy,\bfell_{\alpha_1,\ldots,\alpha_m}\rangle \in \Z \quad \text{for all} \quad \alpha_1+\ldots+\alpha_m\le t.$$

First consider case~\eqref{eq:dualcase1}.
If 
$$u \ge \sum_{i=1}^m \beta_i,
$$ 
then the expression~\eqref{eq:yyy1} becomes 
$$
y_{\beta_1,\ldots,\beta_m}=\frac{1}{a_0^u}  \sum_{\substack{\gamma_{\ell}\le \beta_{\ell} \\ 1 \le \ell \le m}}  (-1)^{\sum_{i=1}^m (j_i - \beta_i)} a_0^{\sum_{i=1}^m \gamma_i } k_{\gamma_1, \cdots, \gamma_m} \prod_{i=1}^{m} \binom{\beta_i}{\gamma_i} a_i^{(\beta_i - \gamma_i)}
$$
where $k_{\gamma_1, \cdots, \gamma_m} \in \mathbb{Z}$.

By our inductive hypothesis, there exists $k_{\gamma_1, \cdots, \gamma_m} \in \mathbb{Z}$ such that for all $(\alpha_1, \cdots, \alpha_m) \in \cI^*_{\beta_1, \cdots, \beta_m}$ we have 
\begin{align*}
y_{\alpha_1, \cdots, \alpha_m} = \frac{1}{a_0^u} \sum_{(\gamma_1, \cdots, \gamma_m) \in \cI_{\alpha_1, \cdots, \alpha_m}} &(-1)^{\sum_{i=1}^m (\alpha_i - \gamma_i)}\\
&a_0^{\sum_{i=1}^m \gamma_i} k_{\gamma_1, \cdots, \gamma_m} \prod_{i=1}^{m} \binom{\alpha_i}{\gamma_i} a_i^{(\alpha_i - \gamma_i)}.
\end{align*}

Since $\vy=\{y_{j_1,\ldots,j_m}\}_{j_1+\ldots+j_m\le t}\in \cL_0^{*}$, we have 
\begin{equation}
\label{eq:ylbeta}
\langle y,\bfell_{\beta_1,\ldots,\beta_m}\rangle =k_{\beta_1,\ldots,\beta_m},
\end{equation}
for some $k_{\beta_1,\ldots,\beta_m}\in \Z$. Let $x_{\alpha_1,\ldots,\alpha_m}$ denote the coordinates of the vector $\bfell_{\beta_1,\ldots,\beta_m}$ when ordered lexicographically, so that~\eqref{eq:ylbeta} implies
\begin{equation}\label{ybeta}
\begin{split}
y_{\beta_1, \cdots, \beta_{m}} = \frac{1}{x_{\beta_1, \cdots, \beta_{m}}}
\Bigl(&k_{\beta_1, \cdots, \beta_{m}} \\
& \quad - \sum_{(\alpha_1, \cdots, \alpha_{m}) \in \cI^*_{\beta_1, \cdots, \beta_{m}}}{x_{\alpha_1, \cdots, \alpha_{m}}} y_{\alpha_1, \cdots, \alpha_{m}}\Bigr).
\end{split}
\end{equation}
By~\eqref{eq:lcoefficients}
$$
x_{\beta_1, \cdots, \beta_{m}} = a_0^{u-\sum_{i=1}^{m} \beta_i}
$$
and
$$
x_{\alpha_1, \cdots, \alpha_{m}} = a_0^{u-\sum_{i=1}^{m} \beta_i} \prod_{i=1}^{m} \binom{\beta_i}{\alpha_i} a_i^{\beta_i-\alpha_i}
$$

Substituting the above into~\eqref{ybeta} and using our inductive hypothesis, we see that
\begin{equation}
\label{eq:y1bths3}
y_{\beta_1, \cdots, \beta_m} = \frac{a_0^{\beta_1 + \cdots + \beta_m}}{a_0^u} k_{\beta_1, \cdots, \beta_m} - S.
\end{equation}
where 
\begin{align*}
S& = \sum_{(\alpha_1, \cdots, \alpha_{m}) \in \cI^*_{\beta_1, \cdots, \beta_{m}}} \prod_{i=1}^{m} \binom{\beta_i}{\alpha_i} a_i^{\beta_i-\alpha_i} \frac{1}{a_0^u}\\
&\qquad  \qquad \qquad \times \sum_{(\gamma_1, \cdots, \gamma_{m}) \in \cI_{\alpha_1, \cdots, \alpha_{m}}} (-1)^{\sum_{j=1}^m (\alpha_j - \gamma_j)} a_0^{\sum_{r=1}^m \gamma_r}\\
&\qquad \qquad \qquad \qquad \qquad  \qquad \qquad \qquad k_{\gamma_1, \cdots, \gamma_m} \prod_{\ell=1}^{m} \binom{\alpha_\ell}{\gamma_\ell} a_\ell^{(\alpha_\ell - \gamma_\ell)}.
\end{align*}

Adding and subtracting the term corresponding to $(\alpha_1, \cdots, \alpha_m) = (\beta_1, \cdots, \beta_m)$ in the outermost summation on the right hand side of the above results in
\begin{equation}\label{sybetas}
\begin{split}
S = T - \sum_{(\gamma_1, \cdots, \gamma_m) \in \cI_{\beta_1, \cdots, \beta_m}} (-1)^{\sum_{j=1}^m (\beta_j - \gamma_j)} & a_0^{\sum_{r=1}^m \gamma_r - u} \\
 k_{\gamma_1, \cdots, \gamma_m}& \prod_{\ell=1}^m \binom{\beta_\ell}{\gamma_\ell} a_\ell^{\beta_\ell - \gamma_\ell},
\end{split} 
\end{equation}
where
\begin{align*}
T&  = \sum_{(\alpha_1, \cdots, \alpha_m) \in \cI_{\beta_1, \cdots, \beta_m}} \prod_{i=1}^m \binom{\beta_i}{\alpha_i} a_i^{\beta_i - \alpha_i} \frac{1}{a_0^u}\\
&\qquad  \qquad \qquad 
\times \sum_{(\gamma_1, \cdots, \gamma_m) \in \cI_{\alpha_1, \cdots, \alpha_m}} (-1)^{\sum_{j=1}^m (\alpha_j - \gamma_j)} a_0^{\sum_{r=1}^m \gamma_r} \\
&\qquad \qquad \qquad \qquad \qquad \qquad \qquad 
\times k_{\gamma_1, \cdots, \gamma_m} \prod_{\ell=1}^m \binom{\alpha_\ell}{\gamma_\ell} a_i^{(\alpha_\ell - \gamma_\ell)}.
\end{align*}
Interchanging the order of summation in $T$, we get:
\begin{align*}
T & = \sum_{(\gamma_1, \cdots, \gamma_m) \in \cI_{\beta_1, \cdots, \beta_m}} (-1)^{\sum_{j=1}^m \gamma_j} \left(\prod_{\ell=1}^{m}a_i^{\beta_\ell - \gamma_\ell}\right) a_0^{\sum_{r=1}^m \gamma_r-u} k_{\gamma_1, \cdots, \gamma_m}\\
& \qquad \qquad \qquad \qquad \quad \times \sum_{\substack{(\alpha_1,\dots,\alpha_m) \\ (\alpha_1, \cdots, \alpha_m) \in \cI_{\beta_1, \cdots, \beta_m} \\ (\gamma_1, \cdots, \gamma_m) \in I_{\alpha_1, \cdots, \alpha_m}}} (-1)^{\sum_{j=1}^m \alpha_i} \prod_{\ell=1}^m \binom{\beta_\ell}{\alpha_\ell}\binom{\alpha_\ell}{\gamma_\ell}.
\end{align*}
Noting the identity:
$$
\binom{\beta}{\alpha} \binom{\alpha}{\gamma} = \binom{\beta}{\gamma} \binom{\beta - \gamma}{\beta - \alpha},
$$
we can further simplify $T$ as follows:
\begin{align*}
T & = \sum_{(\gamma_1, \cdots, \gamma_m) \in \cI_{\beta_1, \cdots, \beta_m}} (-1)^{\sum_{j=1}^m \gamma_j} \left(\prod_{\ell=1}^{m}\binom{\beta_\ell}{\gamma_\ell}a_i^{\beta_\ell - \gamma_\ell}\right) a_0^{\sum_{r=1}^m \gamma_r-u} k_{\gamma_1, \cdots, \gamma_m}\\
& \qquad \qquad \qquad  \times \sum_{\substack{(\alpha_1,\dots,\alpha_m) \\ (\alpha_1, \cdots, \alpha_m) \in \cI_{\beta_1, \cdots, \beta_m} \\ (\gamma_1, \cdots, \gamma_m) \in I_{\alpha_1, \cdots, \alpha_m}}} (-1)^{\sum_{j=1}^m \alpha_i} \prod_{\ell=1}^m \binom{\beta_\ell-\gamma_\ell}{\beta_{\ell}-\alpha_{\ell}}.
\end{align*}

Recalling~\eqref{eq:Idef}, the summation conditions on the innermost sum over $(\alpha_1,\dots,\alpha_m)$ is equivalent to
$$\gamma_\ell\le \alpha_\ell\le \beta_\ell,  \quad 1\le \ell \le m$$
and hence
\begin{equation}
\begin{split}
\label{eq:T11}
T & = \sum_{(\gamma_1, \cdots, \gamma_m) \in \cI_{\beta_1, \cdots, \beta_m}} (-1)^{\sum_{j=1}^m \gamma_j} \left(\prod_{\ell=1}^{m}\binom{\beta_\ell}{\gamma_\ell}a_i^{\beta_\ell - \gamma_\ell}\right) a_0^{\sum_{r=1}^m \gamma_r-u}\\
& \qquad  \qquad  \qquad  \quad \times   k_{\gamma_1, \cdots, \gamma_m} \prod_{\ell=1}^{m}\sum_{\gamma_{\ell}\le \alpha_{\ell}\le \beta_{\ell}} (-1)^{\sum_{j=1}^m \alpha_i} \binom{\beta_\ell-\gamma_\ell}{\beta_{\ell}-\alpha_{\ell}}. \end{split}
\end{equation} 
If $(\gamma_1,\ldots,\gamma_m)\neq (\beta_1,\ldots,\beta_m)$ then there exists some $j$ such that 
$$\gamma_j<\beta_j.$$ 
Hence by the binomial theorem 
$$\sum_{\gamma_j \le \alpha_j \le \beta_j} (-1)^{\sum_{i=1}^m \alpha_i} \prod_{i=1}^m \binom{\beta_i - \gamma_i}{\beta_i - \alpha_i}=0.$$
Therefore, this implies that the only term remaining in summation over $(\gamma_1,\dots,\gamma_m)$ in~\eqref{eq:T11} corresponds to $(\gamma_1,\dots,\gamma_m)=(\beta_1,\dots,\beta_m)$ and hence
$$T=a_0^{\sum_{i=1}^{m}\beta_i-u}k_{\beta_1,\ldots,\beta_m}.$$
Combining the above with~\eqref{eq:y1bths3} and~\eqref{sybetas} completes the proof of the inductive step when~\eqref{eq:dualcase1} holds. The case~\eqref{eq:dualcase2} is similar with a slight modifications to the vectors~\eqref{eq:lcoefficients}. 
\end{proof} 

\subsection{Proof of Theorem~\ref{thm:ApproxGCD-Rigor}}
\label{sec: ApproxGCD-Rigor}   
We next explain how Lemma~\ref{lem:dual} may be applied to remove Assumption~\eqref{ass:heuristic} in Lemma~\ref{lem:approxgcd-alg}. 

Recall that $a_0=pq$ for some primes $p,q$ and assume that $a_1,\ldots,a_m$ satisfy
\begin{equation}
\label{eq:abjj}
a_j=pb_j+r_j, \quad 1\le j \le m, \quad |r_j|\le X_j.
\end{equation}

We   show that if the tuple $(b_1,\ldots,b_m)$ is chosen uniformly at random from the cube $[1,q]^{m}$ then with probability $1+o(1)$, one may remove heuristic Assumption~\ref{ass:heuristic} provided there exists $H=o(p^{1-1/(m+1)})$ such that 
$$X_1=\dots=X_m=H.$$

Taking $u=1$ in Lemma~\ref{lem:dual}, we see that each point 
$$
y=\{y_{j_1,\ldots,j_m}\}_{j_1+\ldots+j_m\le t}\in \cL_0^{*}
$$ 
satisfies 
\begin{equation}
\label{eq:yiii}
\|y_{i_1,\ldots,i_m}\|=\left\|\frac{a_1^{i_1}\ldots a_m^{i_m}n}{a_0}\right\|,
\end{equation}
for some $n\in \Z$. 
Note that 
$$
\left\|\frac{r_1^{i_1}\ldots r_m^{i_m}}{p}\right\|=\left\|\frac{a_1^{i_1}\ldots a_m^{i_m}q}{a_0}\right\|
$$
and hence from Lemma~\ref{lem:dual} there exists some 
$$\widetilde y=\{\widetilde y_{i_1,\ldots,i_m}\}_{i_1+\ldots+i_m\le t}\in \cL_0^{*}$$
such that 
\begin{equation}
\label{eq:tildey}
\widetilde y_{i_1,\ldots,i_m}=\left\|\frac{r_1^{i_1}\ldots r_m^{i_m}}{p}\right\|.
\end{equation}
In particular 
$$
\cL_0^{*}\cap \frac{c_0}{p}B^{*}\neq \{0\}, 
$$
which implies 
$$\lambda_1^{*}\le \frac{c_0}{p}$$
for some absolute constant $c_0$. 

We next show that with probability $1+o(1)$, the point~\eqref{eq:tildey} is the shortest lattice point of $\cL_0^{*}$. This will imply there exists constants $c_0,c_1$ such that with probability $1+o(1)$
$$\frac{c_1}{p}\le \lambda_1^{*}\le \frac{c_0}{p}.$$

Take $y=\{y_{i_1,\ldots,i_m}\}_{i_1+\ldots+i_m\le t}$ as in~\eqref{eq:yiii} and subtract off the closest multiple of $\widetilde y$ defined as in~\eqref{eq:tildey}. The resulting point 
$$
z=\{z_{i_1,\ldots,i_m}\}_{i_1+\ldots+i_m\le t}
$$ 
has coordinates given by
$$z_{i_1,\ldots,i_m}=\left\| \frac{n((b_1p+r_1)^{i_1}\ldots (b_m p+r_m)^{i_m}-r_1^{i_1}\ldots r_m^{i_m})}{a_0}\right\|.$$
Let $T$ count the number of $b_1,\ldots,b_m\le q$ such that there exists $n\in \Z$ with 
\begin{equation}
\label{eq:1}
\begin{split}
\left\|\frac{n((b_1p+r_1)^{i_1}\ldots (b_mp+r_m)^{i_m}-r_1^{i_1}\ldots r_m^{i_m})}{a_0}\right\|\qquad &
\\
\le \frac{CX_1^{i_1}\ldots X_m^{i_m}}{p}&. 
\end{split}
\end{equation}
If we can show that  
\begin{equation}
\label{eq:TT}
T=o(q^{m})
\end{equation} then it  follows that for some absolute constant $c$ we have $\lambda_1^{*}\ge \frac{c}{p}$ with probability $1+o(1)$. Hence from Lemma~\ref{lem:ba} 
\begin{equation}
\label{eq:lbin}
\lambda_{\binom{t+m}{m}}=O(p)
\end{equation} with probability $1+o(1)$. If $a_0,\ldots,a_{m}$ is such that the lattice generated by polynomials~\eqref{eq:fedef} satisfies~\eqref{eq:lbin} then for some absolute constant $C_0$, it is easy to see one may select $m$ lattice points in the intersection 
$$
\cL\cap C_0 pB
$$
which correspond to algebraically independent polynomials, thus removing the heuristic Assumption~\ref{ass:heuristic}.

It remains to establish~\eqref{eq:TT}. Recall our assumption that there exists $H$ satisfying
$$H=X_1=X_2=\ldots=X_m.$$
 Recalling the definition of $\cU^{(m)}_{q}(\vec{X})$ just below~\eqref{eq:fdef}, the fact that $a_0=pq$ and~\eqref{eq:1}, we see that 
$$T=\cU^{(m)}_{q}(\vec{X}).$$
with parameter $\ell=p$ in~\eqref{eq:fdef}. By Theorem~\ref{thm:approxgcd}, 
$$
T\le \left( \frac{q^2 H^6}{p^3}+\frac{qH^2}{p}\right)\left(\frac{qH}{p}\right)^{m-1}\exp{\left(O\left(\frac{\log{q}}{\log\log{q}}\right)\right)}.
$$ 
Hence if
$$
H=O\(p^{1-1/(m+1)-\varepsilon}\)
$$
then~\eqref{eq:TT} is satisfied. This establishes Theorem~\ref{thm:ApproxGCD-Rigor}.


\section{Preliminaries} 

\subsection{Background on character sums}
 We refer to~\cite[Chapter~3]{IwKow} for a 
background on multiplicative characters.
Let $\cX_q$ denote the set of multiplicative characters 
modulo $q$ and let $\chi_0$ denote the principal character.
Denote by $\cX_q^* = \cX_q \setminus \{\chi_0\}$ the set of
non principal characters.


We   use the following bound on moments of character sums due to Ayyad, Cochrane and Zheng~\cite{ACZ}.

\begin{lemma}
\label{lem:CharSum4}
For any integer $H<q$  we have
$$\sum_{\chi \in \cX_q^*}\left|\sum_{1\le y \le H}\chi(y) \right|^{4}
\ll q H^{2}(\log q)^{2}.$$
\end{lemma}

\subsection{Some bounds on arithmetic functions}

In this section we collect some well known bounds on arithmetic functions.  Let as usual  $\varphi(k)$, $\tau(k)$,  $\omega(k)$ denote 
the Euler function, the 
number of positive integer divisors and  the 
number of prime divisors   of  a positive integer $k$, respectively. We also use $\zeta(s)$ to denote the Riemann zeta-function.

Clearly the trivial inequality $\omega(k)! \le k$ and the Stirling 
formula imply that 
\begin{equation}
\label{eq:omega}
\omega(k) \ll \frac{\log k}{\log \log k}, \qquad k \ge 3.
\end{equation}
We also note the corresponding bound for $\tau(k)$, see,
 for example,~\cite[Theorem~317]{HardyWright},
\begin{equation}
\label{eq:tau}
\tau(k)\le \exp\left({O\left(\frac{\log{k}}{\log\log{k}}\right)}\right).
\end{equation}

We recall the following well-know elementary bound, see~\cite{Sita} for a much more precise result.

\begin{lemma}
\label{lem:sumphi} We have
$$
\sum_{1\le z \le Z} \frac{z}{\varphi(z)} =  \frac{315\,\zeta(3)}{2\pi^4} Z + O(\log Z). 
$$
\end{lemma}

Our next result follows immediately from a stronger
and much more general estimate of Shiu~\cite[Theorem~2]{Shiu} (taken with
$r = \lambda = 1$ and
$x=y$),
which in turn is a very special case of~\cite[Theorem~1]{Shiu};
even more general results are given by Nair and Tenenbaum~\cite{NaTe}.
As usual, we use $\varphi(k)$ to denote the Euler function. 

\begin{lemma}\label{lem:tau AP}
For any fixed real  $\varepsilon > 0$, and integers $u,v,Z$ satisfying $Z \ge u^{1+\varepsilon}$
 $\gcd(u,v) = 1$, we have
$$
\sum_{\substack{z \le Z \\  z \equiv v \mod  u}} \omega(z)
 \ll \frac{Z}{\varphi(u)} \log \log Z,
$$
where the implied constant depends only on $\varepsilon$.
\end{lemma}

We now need the following simple statement.

\begin{lemma}
\label{lem:sumdiv} Let $\nu \ge 1$ be a fixed integer. 
There exists a polynomial $Q_{\nu}$ of degree $\nu$ such that 
$$
\sum_{1\le z \le Z} \tau(z^{\nu})=ZQ_{\nu}(\log{Z})+O(Z^{1-1/\nu+o(1)}).
$$
\end{lemma}

\begin{proof} 
The numbers $\tau(z^{\nu})$ are coefficients of the Dirichlet series
$$
F(s)=\prod_{p}\left(1+\frac{\nu+1}{p^s}+\frac{2\nu+1}{p^{2s}}+\ldots \right)=\sum_{z=1}^{\infty}\frac{\tau(z^{\nu})}{z^{s}}.
$$
Note that 
\begin{align*}
F(s)&=\zeta(s)^{\nu+1}\prod_{p}\left(1+\frac{\nu+1}{p^s}+\frac{2\nu+1}{p^{2s}}+\ldots \right)\left(1-\frac{1}{p^s} \right)^{\nu+1} \\
&=\zeta(s)^{\nu+1} \\ & \times\prod_{p}\left(1+\frac{\nu+1}{p^s}+\frac{2\nu+1}{p^{2s}}+\ldots \right)\left(1-\frac{\binom{\nu+1}{1}}{p^s}+\frac{\binom{\nu+1}{2}}{p^{2s}}+\ldots \right) \\ 
&=\zeta(s)^{\nu+1}\prod_{p}\left(1-\frac{\nu(\nu-1)/2}{p^{2s}}+\ldots\right)=\zeta(s)^{\nu+1}F_0(s),
\end{align*}
where $F_0(s)$ is an analytic function represented by a Dirichlet series uniformly convergent in the region $\Re{s}>1/2+\delta$ for any fixed $\delta>0$.

Let the coefficients $a_1,a_2,\ldots$, be defined by
$$F_0(s)=\sum_{n=1}^{\infty}\frac{a_n}{n^s}.$$
We see that the sequence $\{a_n\}_{n=1}^{\infty}$ is supported on the set of squarefull integers and hence
\begin{equation}
\label{eq:anbound}
\sum_{1\le  n \le N}a_n \ll N^{1/2}.
\end{equation}
Let $\tau_{\nu+1}(n)$ denote the coefficients of the Dirichlet series
$$\zeta(s)^{\nu+1}=\sum_{n=1}^{\infty}\frac{\tau_{\nu+1}(n)}{n^s}.$$
By the above we have
$$\tau(z^{\nu})=\sum_{de=z}a_d\tau_{\nu+1}(e),$$
which implies that
$$
\sum_{z=1}^{Z}\tau(n^{\nu})=\sum_{d\le Z}a_d\sum_{e\le Z/d}\tau_{\nu+1}(e).
$$
Since (see, for example,~\cite[Chapter~XII]{Titchmarsh}), for any fixed $\varepsilon > 0$ we have
$$
\sum_{1\le n \le N}\tau_{\nu+1}(n)=NP_{\nu}(\log{N})+O\left(N^{1-1/k+\varepsilon}\right),
$$
where $P_{\nu}$ is a polynomial of degree $\nu$ over $\R$,
we see that
$$
\sum_{z=1}^{Z}\tau(n^{\nu})=Z\sum_{d\le Z}\frac{a_d}{d}P_{\nu}(\log{(Z/d)})+O\left(Z^{1-1/k+\varepsilon}\sum_{d\le Z}\frac{a_d}{d^{1-1/\nu+\varepsilon}} \right),
$$
and since
$$\sum_{d\le Z}\frac{a_d}{d^{1-1/k+\varepsilon}}\le \sum_{d=1}^{\infty}\frac{a_d}{d^{1-1/k+\varepsilon}}=F_0(1-1/\nu+\varepsilon)\ll 1,$$
we get
\begin{equation}
\label{eq:tauv1}
\sum_{z=1}^{Z}\tau(n^{\nu})=Z\sum_{d\le Z}\frac{a_d}{d}P_{\nu}(\log{(Z/d)})+O\left(Z^{1-1/\nu+\varepsilon}\right).
\end{equation}

Since $P_v$ is a polynomial of degree $\nu$, we may write
$$
P_{\nu}(\log{(Z/d)})=\sum_{j=0}^{\nu}\log^{\nu-j}Z\sum_{i\le j}\alpha_{i,j}\log^{j}{d},
$$
for some real numbers $\alpha_{i,j}$. This implies that
\begin{equation}
\label{eq:tauv2}
\sum_{d\le Z}\frac{a_d}{d}P_{\nu}(\log{(Z/d)})=
\sum_{j=0}^{\nu}\log^{\nu-j}{Z}\sum_{i\le j}\alpha_{i,j}\sum_{d\le Z}\frac{a_d\log^{i}{d}}{d}.
\end{equation}
Fixing $0\le i \le \nu$, we write
$$
\sum_{d\le Z}\frac{a_d\log^{i}{d}}{d}=(-1)^{i}F^{(i)}(1)+\sum_{d>Z}\frac{a_d\log^{i}{d}}{d},
$$
and since
$$
\sum_{d>Z}\frac{a_d\log^{i}{d}}{d}\ll \int_{Z}^{\infty}\frac{1}{t^2}\sum_{d\le t}a_d (\log^{i}{d})dt,
$$
we have from~\eqref{eq:anbound}
\begin{equation}
\label{eq:tauv3}
\sum_{d>Z}\frac{a_d\log^{i}{d}}{d}\ll \int_{Z}^{\infty}t^{-(3/2-\varepsilon)}dt\ll Z^{-1/2+\varepsilon}.
\end{equation}
Combining~\eqref{eq:tauv1},~\eqref{eq:tauv2} and~\eqref{eq:tauv3} gives
$$
\sum_{1\le z \le Z} \tau(z^{\nu})=ZQ_{\nu}(Z)+O(Z^{1-1/\nu+\varepsilon}),
$$
for some polynomial $Q_{\nu}$ of degree $\nu$.
\end{proof}

\subsection{Preliminary reductions}
In this section we reduce the problem of bounding $\#\cU^{(m)}_{q}(\vec{X})$ to bounding a simpler system of equations. For a tuple of integers $\vec{X}=\{X_{\vec{e}}\}_{1\le |\vec{e}|\le t}$, 
indexed lexicographically by vectors $\vec{e}$ as in Section~\ref{sec:appl}, 
we let $\cV^{(m)}_{q}(\vec{X})$ denote the set of solutions to the system of equations
$$sy_1^{e_1}\ldots y_{m}^{e_m}\equiv x_{\vec{e}} \mod  q, \quad 1\le |\vec{e}| \le t,$$
in variables $s,y_i,x_{\vec{e}}$ satisfying
\begin{equation}
\label{eq:Vvarcond}
1\le s \le q-1, \quad 1\le y_i \le q-1, \quad  |x_{\vec{e}}| \le X_{\vec{e}}, \quad 1\le |\vec{e}| \le t.
\end{equation}

We first show that we can express the polynomials $\ell^{|\vec{e}|-1}y_1^{e_1}\ldots y_m^{e_m}$ as integer linear combinations of the  $f_{\vec{e}}(y_1,\ldots,y_m)$ with small coefficients.

For each integer vector $\vec{e}_0=(e_{0,1},\ldots,e_{0,m})$ we define the set $\cE(\vec{e}_0)$ by
$$\cE(\vec{e}_0)=\{ \vec{e}=(e_1,\ldots,e_m) \in \Z^m:~ 0\le e_j\le e_{0,j}, \ 1\le j \le m\}.$$ 

\begin{lemma}
\label{lem:linear}
For each $\vec{\vec{e}_0}$ with $1\le |\vec{e}_0|\le t$ there exists coefficients $c_{\vec{e},\vec{e}_0}$  such that 
\begin{equation}
\label{eq:linear}
\ell^{|\vec{e}_0|-1}y_1^{e_{0,1}}\ldots y_m^{e_{0,m}}=\sum_{\vec{e}\in \cE(\vec{e}_0)}c_{\vec{e},\vec{e}_0}f_{\vec{e}}(y_1,\ldots,y_m),
\end{equation}
and
\begin{equation}
\label{eq:bound}
c_{\vec{e}_0,\vec{e}_0}=1, \qquad c_{\vec{e},\vec{e}_0}\ll r_1^{e_{0,1}-e_1}\ldots r_{m}^{e_{0,m}-e_m},
\end{equation}
where the implied constant depends only on $t$ and $m$. 
\end{lemma}

\begin{proof}
We proceed by induction on $|\vec{e}_0|$.

The case $|\vec{e}_0|=1$ is trivial since then there exists some $i$ such that $f_{\vec{e}_0}(y_1,\ldots,y_m)$ has the form
$$f_{\vec{e}_0}(y_1,\ldots,y_m)=y_i.$$

Let $|\vec{e}_0|\ge 2$  and suppose for each $|\vec{e}| \le |\vec{e}_0|-1$ there exists $c_{\vec{e_1},\vec{e}}$ satisfying~\eqref{eq:linear} and~\eqref{eq:bound}. In particular we have
$$
\ell^{|e|-1}y_1^{e_1}\ldots y_{m}^{e_m}=\sum_{\vec{e}_1\in \cE(\vec{e})}c_{\vec{e_1},\vec{e}}f_{\vec{e}_1}(y_1,\ldots,y_m).
$$
Recalling the definition of $f_{\vec{e}_0}(y_1,\ldots,y_m)$ in~\eqref{eq:fdef} have
\begin{align*}
f_{\vec{e}_0}(y_1,\ldots,y_m)&=\frac{\prod_{i=1}^{m}(\ell y_i+r_i)^{e_{0,i}}-\prod_{i=1}^{m}r_i^{e_{0,i}}}{\ell} \\ 
&=\ell^{|\vec{e}_0|-1}y_1^{e_{0,1}}\ldots y_m^{e_{0,m}} \\ & \quad +\sum_{\substack{\vec{e}\in \cE(\vec{e}_0) \\ 0<|\vec{e}|<|\vec{e}_0|}}\left(\prod_{i=1}^{m}\binom{e_{0,i}}{e_i}r_i^{e_{0,i}-e_i}\right)\ell^{|\vec{e}|-1}\prod_{i=1}^{m}y_i^{e_i}.
\end{align*}
By our inductive assumption, for each $1\le |\vec{e}| \le |\vec{e}_0|-1$ we have
$$\ell^{|\vec{e}|-1}\prod_{i=1}^{m}y_i^{e_i}=\sum_{\vec{e}_1\in \cE(\vec{e})}c_{\vec{e_1},\vec{e}}f_{\vec{e}_1}(y_1,\ldots,y_m),$$
with coefficients $c_{\vec{e}_1,\vec{e}}$ satisfying for each $\vec{e}_1\in \cE(\vec{e})$
\begin{equation}
\label{eq:e1bound}
c_{\vec{e},\vec{e}}=1, \quad c_{\vec{e_1},\vec{e}}\ll r_1^{e_{1}-e_{1,1}}\ldots r_{m}^{e_{m}-e_{1,m}},
\end{equation}
where $\vec{e}_1=(e_{1,1},\ldots,e_{1,m})$.
Substituting this into the above gives
\begin{align*}
\ell^{|\vec{e}_0|-1}&y_1^{e_{0,1}}\ldots y_m^{e_{0,m}}\\
&=f_{\vec{e}_0}(y_1,\ldots,y_m) \\ &\qquad -\sum_{\substack{\vec{e}\in \cE(\vec{e}_0) \\ 0<|\vec{e}|<|\vec{e}_0|}}\left(\prod_{i=1}^{m}\binom{e_{0,i}}{e_i}r_i^{e_{0,i}-e_i}\right)\sum_{\vec{e}_1\in \cE(\vec{e})}c_{\vec{e}_1,\vec{e}}f_{\vec{e}_1}(y_1,\ldots,y_m) \\
&=\sum_{\vec{e}_1\in \cE(\vec{e}_0)}c_{\vec{e}_1,\vec{e}_0}f_{\vec{e}_1}(y_1,\ldots,y_m),
\end{align*}
where the coefficients $c_{\vec{e}_1,\vec{e}_0}$ are given by 
$$
c_{\vec{e}_0,\vec{e}_0}=1 \quad \text{and} \quad c_{\vec{e}_1,\vec{e}_0}=-\sum_{\substack{\vec{e} \\ \vec{e}_1 \in \cE(\vec{e})}}c_{\vec{e}_1,\vec{e}}\left(\prod_{i=1}^{m}\binom{e_{0,i}}{e_i}r_i^{e_{0,i}-e_i}\right).
$$
By~\eqref{eq:e1bound} we have 
\begin{align*}
c_{\vec{e}_1,\vec{e}_0}&\ll \sum_{\substack{\vec{e} \\\vec{e_1} \in \cE(\vec{e})}}\prod_{i=1}^{m}r_i^{e_i-e_{1,i}}\prod_{i=1}^{m}\binom{e_{0,i}}{e_i}r_i^{e_{0,i}-e_i} \\
& \ll \sum_{\substack{\vec{e} \\ \vec{e_1}\in \cE(\vec{e})}}\prod_{i=1}^{m}\binom{e_{0,i}}{e_i}r_i^{e_{0,i}-e_{1,i}}\ll \prod_{i=1}^{m}r_i^{e_{0,i}-e_{1,i}},
\end{align*}
which completes the proof.
\end{proof}

We now establish our main technical tool. 

\begin{lemma}
\label{lem:main-red}
With notation as above, suppose that  whenever there exists some $1\le i_0 \le m$ such that if
$$\vec{e}_1=\{e_{1,1},\ldots,e_{1,m}\} \mand  \vec{e}_2=\{e_{2,1},\ldots,e_{2,m}\},$$
 satisfy
 $$e_{2,j}=\begin{cases} e_{1,j}, \quad \quad \quad \quad \  \quad j\neq i_0, \\ e_{2,j}=e_{1,j}+1, 
 \quad j=i_0, \end{cases}
$$
  we have 
\begin{equation}
\label{eq:ParameterInequality}
\frac{X_{\vec{e}_2}}{X_{\vec{e}_1}}\gg r_{i_0}.
\end{equation}
Then there exists a constant $c$ depending only on $t$ and $m$ such that
$$\#\cU^{(m)}_q(\vec{X})\ll \#\cV^{(m)}_q(c\vec{X}).$$
\end{lemma}

\begin{proof}
We first note that $\cV^{(m)}_q(\vec{X})$ is equal to the set of solutions
 to the system of equations
$$\ell^{|\vec{e}|-1}sy_1^{e_1}\ldots y_m^{e_m}\equiv x_{\vec{e}} \mod  q, \quad 1\le |\vec{e}| \le t,$$
with variables satisfying~\eqref{eq:Vvarcond}.
This can be seen via the change of variables $s\rightarrow \ell^{-1}s$ and $y_i \rightarrow \ell y_i$.

Suppose $$(s,y_1,\ldots,y_m,\{x_{\vec{e}}\}_{1\le |\vec{e}|\le t})\in \cU^{(m)}_q(\vec{X}).$$ 
We recall that $\cU^{(m)}_q(\vec{X})$ denotes the set of solutions to the system of equations
\begin{equation}
\label{eq:Udef1}
sf_{\vec{e}}(y_1,\ldots,y_m)\equiv x_{\vec{e}} \mod  q, \quad 1\le |\vec{e}|\le t,
\end{equation}
with variables satisfying 
\begin{equation}
\label{eq:Uvarcond1}
1\le s \le q-1, \quad 1\le y_i \le q-1, \quad |x_{\vec{e}}|\le X_{\vec{e}}.
\end{equation}
Fixing some $\vec{e}_0$ with $1\le |\vec{e}_0|\le t$, by Lemma~\ref{lem:linear} we have
$$
\ell^{|\vec{e}_0|-1}y_1^{e_{0,1}}\ldots y_m^{e_{0,m}}=\sum_{\vec{e}\in \cE(\vec{e}_0)}c_{\vec{e},\vec{e}_0}f_{\vec{e}}(y_1,\ldots,y_m),
$$
for some constants $c_{\vec{e},\vec{e}_0}$ satisfying for each $\vec{e}\in \cE(\vec{e}_0)$ 
\begin{equation}
\label{eq:cprop}
c_{\vec{e}_0,\vec{e}_0}=1, \quad c_{\vec{e},\vec{e}_0}\ll r_1^{e_{0,1}-e_1}\ldots r_{m}^{e_{0,m}-e_m}.
\end{equation}

Combining the above with equations~\eqref{eq:Udef1} and~\eqref{eq:Uvarcond1} gives
\begin{equation}
\label{eq:UVtrans1}
\ell^{|\vec{e}_0|-1}sy_1^{e_{0,1}}\ldots y_m^{e_{0,m}}\equiv \sum_{\vec{e} \in \cE(\vec{e}_0)}c_{\vec{e},\vec{e}_0}x_{\vec{e}}.
\end{equation}
Let
\begin{equation}
\label{eq:zdef1}
z_{\vec{e}_0}=\sum_{\vec{e} \in \cE(\vec{e}_0)}c_{\vec{e},\vec{e}_0}x_{\vec{e}},
\end{equation}
so that by~\eqref{eq:cprop}
\begin{equation}
\label{eq:zUpperBound}
z_{\vec{e}_0}\ll \sum_{\vec{e} \in \cE(\vec{e}_0)}\prod_{i=1}^{m}r_i^{e_{0,i}-e_{i}}X_{\vec{e}}.
\end{equation}
Fixing some $\vec{e}$ in the above sum and considering the term
$$\prod_{i=1}^{m}r_i^{e_{0,i}-e_{i}}X_{\vec{e}},$$
we may choose a sequence of vectors $\vec{g}_1,\ldots,\vec{g}_R$ such that 
$$\vec{g}_1=\vec{e}, \qquad \vec{g}_R=\vec{e}_0,$$
and for each $1\le j \le R-1$ we have 
$$|\vec{g}_{j+1}-\vec{g}_j|=1.$$
By~\eqref{eq:ParameterInequality} this implies that
$$
\prod_{i=1}^{m}r_i^{e_{0,i}-e_{i}}\le \prod_{j=1}^{R-1}\frac{X_{\vec{g}_{j+1}}} {X_{\vec{g}_j}}
= \frac{X_{\vec{e}_0}}{X_{\vec{e}}},
$$
and hence 
$$\prod_{i=1}^{m}r_i^{e_{0,i}-e_{i}}X_{\vec{e}}\ll X_{\vec{e}_0},$$
which substituted into~\eqref{eq:zUpperBound} gives
\begin{equation}
\label{eq:zUB1}
z_{\vec{e}_0}\ll X_{\vec{e}_0}.
\end{equation}
By~\eqref{eq:UVtrans1}, \eqref{eq:zdef1}, \eqref{eq:zUB1} and the fact that $\vec{e}_0$ is arbitrary 
$$(s,y_1,\ldots,y_m,\{ z_\vec{e}\}_{1\le |\vec{e}|\le t})\in \cV^{(m)}_q(c\vec{X}).$$
Since  the numbers $s$ and $y_1,\ldots,y_m$ uniquely determine each $x_{\vec{e}}$,  
for some choice of integers $\{a_{\vec{e}}\}_{1\le |\vec{e}|\le t}$ we have 
$$\#\cU^{(m)}_q(\vec{X})\le \#\cV^{(m)}_q(c\vec{X})N(\{a_{\vec{e_0}}\}_{1\le |\vec{e}_0| \le t}),$$
where $N(\{a_{\vec{e_0}}\}_{1\le |\vec{e}_0| \le t})$ denotes 
 the number of solutions to the system of equations
\begin{equation}
\label{eq:Na}
\sum_{\vec{e}\in \cE(\vec{e}_0)}c_{\vec{e},\vec{e}_0}x_{\vec{e}}\equiv a_{\vec{e}_0} \mod  q, \quad  |x_{\vec{e}}|\le X_{\vec{e}}, \quad 1\le |\vec{e}_0| \le t.
\end{equation}
We now show that $N(\{a_{\vec{e}_0}\}_{1\le |\vec{e}_0| \le t})\le 1$ and thus the numbers $x_{\vec{e}}$ uniquely  determine the numbers $z_{\vec{e}_0}$ in~\eqref{eq:zdef1}. 

We order the indicies of the variables 
occuring in~\eqref{eq:Na} lexicographically. Considering the matrix $M$ corresponding to the 
equations~\eqref{eq:Na} with this ordering of variables. We see that $M$ has entries equal to $0$ below the main diagonal and by~\eqref{eq:bound}, $M$ has each diagonal entry equal to $1$. This implies that $M$ has nonzero determinant 
and hence the system~\eqref{eq:Na} has at most one solution, so that
$$\#\cU^{(m)}_q(\vec{X})\le \#\cV^{(m)}_q(c\vec{X})$$
which concludes the proof. 
\end{proof}

\section{Proofs of main results} 

\subsection{Proof of Theorem~\ref{thm:main1}}

We consider only the first three congruences from~\eqref{eq:congrf};
\begin{equation}
\label{eq:UV}
\# \cU_{n,q}(\vec{a}, \vec{k}, \vec{h}) \le V_{n,q}(\vec{a}, \vec{k}, \vec{h}),
\end{equation}
where $V_{n,q}(\vec{a}, \vec{k}, \vec{h})$ is the number of $u \in \F_q^*$ 
for which the system of congruences 
\begin{equation}
\label{eq:3congr}
su^{k_i} \equiv y_j \mod  q, \qquad   s \in \F_q^*, \ |y_i|\le h_i, 
\quad j = 1,2,3,
\end{equation}
has a solution.
We now find a non-zero integer solution $\vec{r} = (r_1, r_2, r_3)$ of smallest 
Euclidean norm $\|\vec{r}\|$ to the 
following system of equations
$$
r_1+r_2+r_3 = 0 \mand k_1r_1+k_2r_2+k_3r_3 = 0.
$$
Clearly $\vec{r}$  has an interpretation as the shortest 
vector of a certain 3-dimesional lattice and thus by the  
Minkowski's,  see~\cite[Theorem~5.3.6]{GrLoSch},
 we have $\|\vec{r}\| = O(1)$. 
Furthermore, since $k_1, k_2, k_3$ are pairwise distinct, 
the condition $\vec{r} \ne 0$ immediately implies $r_1r_2r_3\ne 0$. 

Using this vector  $\vec{r}$ we derive that for every 
solution $(s,y_1, y_2, y_3)$ to~\eqref{eq:3congr} we have 
\begin{equation}
\label{eq: prod congr r}
\lambda y_1^{r_1} y_2^{r_2} y_3^{r_3} \equiv 1 \mod  q, \qquad   |y_i|\le h_i, \ 
i =1,2,3,
\end{equation}
with $\lambda  \equiv a_1^{-r_1} a_2^{-r_2} a_3^{-r_3}  \mod  q$.

Since the congruence~\eqref{eq: prod congr r} does not depend on $u$
and when $(y_1, y_2, y_3)$ are fixed there are clearly only $O(1)$ 
corresponding values of $u$ (as~\eqref{eq:3congr} implies 
$u^{k_2-k_1} \equiv y_2/y_1 \mod  q$), we see from~\eqref{eq:UV}
that 
\begin{equation}
\label{eq:UQ}
\# \cU_{n,q}(\vec{a}, \vec{k}, \vec{h}) \ll J,
\end{equation}
where $J$ is the number of solutions to the 
congruence~\eqref{eq: prod congr r}.

Using the orthogonality of characters, we can express $W$ via
the following sum
\begin{align*}
J &= \sum_{|y_1|\le h_1}\sum_{|y_2|\le h_2} \sum_{|y_3|\le h_3}
\frac{1}{q-1}\sum_{\chi \in \cX_q} \chi\(\lambda  y_1^{r_1} y_2^{r_2} y_3^{r_3}\) \\
& = 
\frac{1}{q-1}\sum_{\chi \in \cX_q} \chi\(\lambda\) \prod_{i=1}^3 \sum_{|y_i|\le h_i} \chi^{r_i}\(y_i\).
\end{align*}
By the above properties of $\vec{r}$ we see that there are at most 
$|r_1| + |r_2| + |r_3| = O(1)$ characters $\chi \in \cX_q$ for which $\chi^{r_i} = \chi_0$ 
for some $i =1,2,3$. We estimate the above sums for all such characters trivially,
so denoting by $\Psi$ the set of remaining characters we obtain  
\begin{equation}
\label{eq:W Psi}
J \ll \frac{h_1h_2h_3}{q}+  \frac{1}{q-1}\sum_{\chi \in \Psi}
 \prod_{i=1}^3  \left| \sum_{|y_i|\le h_i} \chi^{r_i}\(y_i\)\right|.
\end{equation}

By H{\"o}lder's inequality
\begin{equation}
\label{eq:SST}
\sum_{\chi \in \Psi} \prod_{i=1} \left| \sum_{|y_i|\le h_i} \chi^{r_i}\(y_i\)\right| \le  S_1^{1/4}S_2^{1/4}T^{1/2}, 
\end{equation}
where 
$$
S_i =\sum_{\chi \in \Psi}\left|\sum_{|y|\le  h_i}\chi^{r_i}\(y\) \right|^4, 
\  i =1,2, \quad \text{and} \quad 
T=\sum_{\chi \in \Psi}\left|\sum_{|y|\le  h_3}\chi^{r_3}\(y\) \right|^2. 
$$

Cleary, when $\chi$ runs through the set $\Psi$, the character 
$\chi^{r_i}$ runs through some subset of $\cX_q^*$ repeating 
each character at most $r_i$ times. Hence,  applying Lemma~\ref{lem:CharSum4}, we derive 
\begin{equation}
\label{eq:Bound Si}
S_i \ll \sum_{\chi \in \cX_q^*}\left|\sum_{|y|\le  h_i}\chi \(y\) \right|^4
\le q h_i^2 (\log q)^2, 
\qquad  i =1,2.
\end{equation}
A similar argument, combined with the orthogonality 
of characters implies
\begin{equation}
\label{eq:Bound T}
T \ll \sum_{\chi \in \cX_q^*}\left|\sum_{|y|\le  h_3}\chi \(y\) \right|^2
\le  \sum_{\chi \in \cX_q}\left|\sum_{|y|\le  h_3}\chi\(y\) \right|^2 = qh_3.
\end{equation}
Substituting~\eqref{eq:Bound Si} and~\eqref{eq:Bound T} in~\eqref{eq:SST}
and recalling~\eqref{eq:UQ} and~\eqref{eq:W Psi},  we obtain 
\begin{equation}
\label{eq:fin}
\# \cU_{n,q}(\vec{a}, \vec{k}, \vec{h})\ll 
\frac{h_1h_2h_3}{q}+(h_1h_2h_3)^{1/2}\log q.
\end{equation}
Clearly, this bound is nontrivial only if $h_1h_2h_3 \le q^2$, in which 
case the second term on the right hand side of~\eqref{eq:fin}
always dominates and the  desired result follows. 

\subsection{Proof of Theorem~\ref{thm:main12}}

The proof is essentially identical to that of Theorem~\ref{thm:main1}.
We only note that in this case we have $k_1=1, k_2=2, k_3=3$, so we can use
$r_1=1, r_2=-2, r_3=1$.  Hence our equations become
\begin{equation}
\label{eq:congy123}
y_1 y_3 \equiv y_2^{2} \mod  q, \qquad 0 < |y_i| < h_i, \ i =1,2,3.
\end{equation}
Clearly $|y_1 y_3 - y_2^{2}| \le H$ where $H = h_1h_3 + h_2^2$.
Thus~\eqref{eq:congy123} implies that $y_1 y_3 = y_2^{2} + zq$ 
for some integer $z$ with $|z|\le H/q$. Hence $z$ can take at most $O(H/q+1)$
possible values and then for each fixed $y_2$, in $O(h_2)$ possible ways, 
we have from~\eqref{eq:tau} that $y_1$ and $y_3$ can take at most  $\exp\left({O\left(\frac{\log{q}}{\log\log{q}}\right)}\right)$ 
possible values.

\subsection{Proof of Theorem~\ref{thm:main2}}

Recalling the proof of Theorem~\ref{thm:main1}, 
in particular~\eqref{eq: prod congr r}, we see that it is enough to estimate 
$W_{n,q}(\vec{a}, \vec{k}, \vec{h})$ on average over $q \in \cQ$.
We also note that $\lambda$ in~\eqref{eq: prod congr r} is a
rational number with the numerator  and denominator of size $A^{O(1)}$.
 
Changing the order of summation, we obtain
\begin{align*}
\sum_{q\in \cQ}\# \cU_{n,q}(\vec{a}, \vec{k}, \vec{h}))&\ll \sum_{q\in \cQ}
\sssum_{\substack{{|y_1|\le h_1,\, |y_2|\le h_2,\, |y_3|\le h_3}
\\ \lambda y_1^{r_1} y_2^{r_2} y_3^{r_3} \equiv 1 \mod  q}}1\\
& =\sssum_{|y_1|\le h_1, \, |y_2|\le h_2,\, |y_3|\le h_3}
\sum_{\substack{q\in \cQ\\ \lambda y_1^{r_1} y_2^{r_2} y_3^{r_3} \equiv 1 \mod  q}}1\\
& = \Sigma_1 + \Sigma_2, 
\end{align*}
where $\Sigma_1$ is the contribution from terms with $\lambda y_1^{r_1} y_2^{r_2} y_3^{r_3} = 1$
and $\Sigma_2$ is the contribution from all other terms.

Considering $\Sigma_1$, for a vector $\vec{y} = (y_1,y_2,y_3)$ with $\lambda y_1^{r_1} y_2^{r_2} y_3^{r_3} = 1$
we estimate the inner sum trivially as $\# \cQ$. 
It is easy to see that the above equation is equivalent to  a relation 
of the type 
\begin{equation}
\label{eq:z123}
a z_1^{s_1}z_2^{s_2} = b z_3^{s_3}, \qquad 0 < |z_i| < J_i, \ i =1,2,3,
\end{equation}
with some relatively prime integers $a,b = A^{O(1)}$,
where $(s_1, s_2, s_3)$ and $(J_1,J_2,J_3)$ are 
permutations of  $(|r_1|, |r_2|, |r_3|)$ and $(h_1,h_2,h_3)$,
respectively. Thus, fixing $z_3$ in $O(J_3) = O(h_3)$ 
possible ways, we see that $z_1^{s_1}$ and $z_2^{s_2}$ run through 
the divisors of $|b z_3^{s_3}|$. 
It is also clear that we can assume that 
$$
J_3 \le J_1^{s_1}J_2^{s_2} = J_2^{O(1)},
$$
in the equation~\eqref{eq:z123}.
We see that the total contribution 
from all such terms in $\Sigma_1$ can be bounded by
$$
\Sigma_1 \ll \#\cQ\sum_{1\le z_3 \le h_3}\tau(bz_3^{s_3})\ll \#\cQ\tau(b)\sum_{1\le z_3\le h_3}\tau(z_3^{s_3}),
$$
so that combining~\eqref{eq:tau}, the bound $b=A^{O(1)}$ and Lemma~\ref{lem:sumdiv} gives
\begin{equation}
\label{eq:Sigma1}
\Sigma_1 \ll \#\cQ h_3\exp\(O\(\frac{\log A}{\log\log A}  + \log \log h_3\)\).
\end{equation}

To estimate $\Sigma_2$ we note that if  $\lambda y_1^{r_1} y_2^{r_2} y_3^{r_3} \equiv 1 \mod  q$
then $q$ divides the numerator of $\lambda y_1^{r_1} y_2^{r_2} y_3^{r_3} - 1$, 
which is a non-zero integer of size $h_3^{O(1)}$ so has at most 
$O\(\log h_3/\log \log h_3\)$ prime divisors. 
The total contribution 
from all such terms can be bounded as 
\begin{equation}
\label{eq:Sigma2}
\Sigma_2 \ll  h_1h_3h_3 \log h_3/\log \log h_3.
\end{equation}
Combining~\eqref{eq:Sigma1} and~\eqref{eq:Sigma2},
we obtain the result.

\subsection{Proof of Theorem~\ref{thm:main3}}

Arguing as in the proof of Theorem~\ref{thm:main2},
 we have
\begin{align*}
\sum_{q\in \cQ}\#U_q(\mathbf{h})&\ll \sum_{q\in \cQ}\sum_{1 \le |y_1|\le h_1}\sum_{1\le |y_2|\le h_2}\sum_{1 \le |y_3|\le h_3}\sum_{\substack{q\in \cQ \\ y_1 y_3\equiv y^2_2 \mod  q}}1 \\
&=\Sigma_1+\Sigma_2,
\end{align*}
where $\Sigma_1$ is the contribution from terms $y_1,y_2,y_3$ with $y_1y_3=y_2^2$ and $\Sigma_2$ the contribution from the remaining terms. Considering $\Sigma_1$, for each value of $y_2$ there are at most $\tau(y_2^2)$ values of $y_1,y_3$ and 
hence by Lemma~\ref{lem:sumdiv}
$$\Sigma_1\ll \# \cQ \sum_{y_2\le h_2}\tau(y_2^2)\ll  \# \cQ  h_2\log^2{h_2}.$$

We now consider $\Sigma_2$. Without loss of generality, we can assume that $h_1\ge h_3$. 
Then we write 
$$
\Sigma_2 \le  \sum_{1\le |y_3|\le h_3} \sum_{1\le  |y_2|\le h_2} 
\sum_{1\le |y_1|\le h_1} \omega(|y_1 y_3 -  y^2_2|).
$$
We set $D = \min\{|h_3|, |h_2|^2\}$ and for each positive integer $d \le D$ we group together 
pairs $(y_2, y_3)$ with $\gcd(y_2^2, y_3) = d$ in a set $\cY_d$.

Now for each pair  $(y_2, y_3) \in \cY_d$ we 
estimate the inner sums via Lemma~\ref{lem:tau AP} with $Z =  |y_3| h_1/d$,  $u= |y_3|/d$
(thus $M\ge m^2$ and we can take $\varepsilon =1$) and $v = |y_2|^2/d$, and where $y_1y_3/d$ plays the
role of $z$. Using the subadditivity of the prime divisor function,  we obtain
\begin{align*}
\sum_{1\le |y_1|\le h_1} \omega(|y_1 y_3 -  y^2_2|) & \le \omega (d) h_1 
+ \sum_{1\le |y_1|\le h_1} \omega(|y_1 y_3/d -  y^2_2/d|)\\
& \ll \omega (d) h_1 + h_1 \frac{|y_3|}{d \varphi( |y_3|/d)} \log \log h_1.
\end{align*}

If  $(y_2, y_3) \in \cY_d$ then $y_2$ belongs to a set of at most $h_2 d^{-1/2}$ integers. 
Hence, writing $y_3 = \pm  w d$ and the extending the summation over all  positive integers $w  \le h_3/d$ we obtain   
\begin{align*}
\Sigma_2 & \ll  h_1 h_2  h_3 
\sum_{1\le d \le D} \omega (d)  d^{-3/2} \\
& \qquad \qquad  \quad +
h_1 h_2 \sum_{1\le d \le D} d^{-1/2}  \sum_{1 \le |w| \le h_3/d} 
 \frac{w}{\varphi(w)} \log \log h_1. 
\end{align*}
Using Lemma~\ref{lem:sumphi} we easily obtain 
$$
\Sigma_2 \ll  h_1 h_2  h_3   \log \log h_1, 
$$
and conclude the proof. 
%

\subsection{Proof of Theorem~\ref{thm:approxgcd}}
By assumption our parameters satisfy the conditions of Lemma~\ref{lem:main-red}, hence it is sufficient to prove the corresponding bound for $\#\cV^{(m)}(\vec{X})$. Fixing a solution to the equation
$$sy_1^{i}\equiv x_{1,i} \mod  q,$$
in variables $s,y_1,x_{1,1},\ldots,x_{1,t}$ there are at most 
$$\left(\frac{qH}{\ell}\right)^{m-1},$$
solutions in remaining variables. Hence with notation as in Theorem~\ref{thm:main12} we have
$$\#\cU_q^{(m)}(\vec{X})\ll \left(\frac{qH}{\ell}\right)^{m-1}\#\cU_{t,q}(\mathbf{h}),$$
where 
$$h_i=\frac{qH^{i}}{\ell}.$$
This implies that 
$$
\#\cU_q^{(m)}(\vec{X})\ll \left(\frac{q^2H^6}{\ell^3}+\frac{qH^2}{\ell}\right)\left(\frac{qH}{\ell}\right)^{m-1},
$$
and concludes the proof.

\subsection{Proof of Theorem~\ref{thm:approxgcdav}}
As in the proof of Theorem~\ref{thm:approxgcd} we have
$$
\frac{1}{\#\cQ}\sum_{q\in \cQ}\#\cU_q^{(m)}(\vec{X})\ll \left(\frac{QH}{\ell}\right)^{m-1}\frac{1}{\#\cQ}\sum_{q\in \cQ}\#\cU_{t,q}(\mathbf{h}),
$$
hence by Theorem~\ref{thm:main3}
\begin{align*}
\frac{1}{\#\cQ}\sum_{q\in \cQ}&\#\cU_q^{(m)}(\vec{X})\\
&\ll \left(\frac{QH}{\ell}\right)^{m-1}\left( 
\frac{QH^2}{\ell}\log^2{Q} 
+\frac{1}{\#\cQ}\frac{Q^3H^6}{\ell^3}\frac{\log Q}{\log\log Q}\right),
\end{align*}
and the result follows. 

\section*{Acknowledgement}

The authors are grateful to Nadia Heninger and Damien Stehl{\'e} for introducing them to this 
problem.

During the preparation of this work, B.~Kerr was supported by the ARC Grants DE220100859 and DP230100534  
and  I.~E.~Shparlinski   by the ARC Grant DP230100530 and  DP230100534

\end{document}